\documentclass[11pt]{amsart}
\pdfoutput=1
\theoremstyle{plain}
\newtheorem{thm}{Theorem}
\newtheorem{prop}{Proposition}
\newtheorem{cor}{Corollary}
\newtheorem{lem}{Lemma}
\theoremstyle{definition}
\newtheorem{defi}{Definition}
\newtheorem{rem}{Remark}
\newtheorem{rems}{Remarks}
\newtheorem{qu}{Question}


\usepackage{hyperref}
\hypersetup{%
	pdftitle={More on Groups and Counter Automata},%
	pdfkeywords={word problem; G-automaton; abelian group}}
\usepackage{amssymb}
\usepackage[abbrev, backrefs, msc-links]{amsrefs}
\usepackage{amsthm}
\usepackage{cleveref}
\usepackage{mleftright}
\usepackage{microtype}
\usepackage{tikz}
\usepackage{color}
\usepackage{enumitem}
\setlist[enumerate]{label = \textup{(\arabic*)}, ref = \textup{(\arabic*)}}

\theoremstyle{plain} 

\crefname{defi}{Definition}{Definitions}
\crefname{thm}{Theorem}{Theorems}
\crefname{lem}{Lemma}{Lemmata}
\crefname{cor}{Corollary}{Corollaries}
\crefname{rem}{Remark}{Remarks}
\crefname{qu}{Question}{Questions}

\DeclareMathOperator{\WP}{WP}
\newcommand{\source}{\mathop{\mathsf{s}}}
\newcommand{\target}{\mathop{\mathsf{t}}}
\newcommand{\Label}{\ell}
\newcommand{\subword}{\sqsubseteq}
\newcommand{\scsubword}{\sqsubseteq_{\mathrm{sc}}}
\newcommand{\pinit}{p_{\mathrm{init}}}
\newcommand{\pter}{p_{\mathrm{ter}}}

\newcommand{\Z}{\mathbb{Z}}
\newcommand{\len}[1]{\lvert #1 \rvert}
\newcommand{\card}[1]{\lvert #1 \rvert}
\newcommand{\setmid}{\mathrel{}\middle|\mathrel{}}

\begin{document}

\title[More on Groups and Counter Automata]
{More on Groups and Counter Automata}
\author[Takao Yuyama]{Takao Yuyama}	
\address{Department of Mathematics, Tokyo Institute of Technology, Tokyo, Japan}	
\email{yuyama.t.aa@m.titech.ac.jp}  
\thanks{This work was supported by JSPS KAKENHI Grant Number 20J23039}	




\begin{abstract}
	Elder, Kambites, and Ostheimer showed that if the word problem of a finitely generated group \(H\)
	is accepted by a \(G\)-automaton for an abelian group \(G\), then \(H\) is virtually abelian.
	We give a new, elementary, and purely combinatorial proof to the theorem.
	Furthermore, our method extracts an explicit connection between the two groups \(G\) and \(H\) from the automaton
	as a group homomorphism from a subgroup of \(G\) onto a finite index subgroup of \(H\).
\end{abstract}

\maketitle

\section{Introduction}

For a group \(G\), a \emph{\(G\)-automaton} is a finite automaton augmented with a register that stores an element of \(G\).
Such an automaton first initializes the register with the identity element \(1_G\) of \(G\)
and may update the register content by multiplying by an element of \(G\) during the computation.
The automaton accepts an input word if the automaton can reach a terminal state and the register content is \(1_G\)
when the entire word is read.
(For the precise definition, see \Cref{subsection-G-automata}.)
For a positive integer \(n\), \(\Z^n\)-automata are the same as \emph{blind \(n\)-counter automata},
which were defined and studied by Greibach~\cites{MR411257, MR513714}.

The notion of \(G\)-automata is discovered repeatedly by several different authors.
The name ``\(G\)-automaton'' is due to Kambites~\cite{MR2259632}.
(In fact, they defined the notion of \(M\)-automata for any monoid \(M\).)
Dassow--Mitrana~\cite{MR1809380} and Mitrana--Stiebe~\cite{MR1807922} use \emph{extended finite automata} (EFA) over \(G\)
instead of \(G\)-automata.

For a finitely generated group \(G\), the \emph{word problem} of \(G\), with respect to a fixed finite generating set of \(G\),
is the set of words over the generating set representing the identity element of \(G\)
(see \Cref{subsection-word-problem} for the precise definitions).
For several language classes, the class of finitely generated groups whose word problem is in the class is determined~\citelist{
	\cite{MR312774}
	\cite{MR710250}
	\cite{MR0357625}
	\cite{MR130286}
	\cite{MR2483126}
	\cite{MR2470538}
	\cite{MR3375036}
},
and many attempts are made for other language classes~\citelist{
	\cite{MR3200359}
	\cite{MR3394671}
	\cite{MR2132375}
	\cite{MR4388995}
	\cite{kanazawa-salvati-2012-mix}
	\cite{MR2274726}
	\cite{MR4000593}
	\cite{MR3871464}
}.
One of the most remarkable theorems about word problems is the well-known result due to Muller and Schupp~\cite{MR710250},
which states that, with the theorem by Dunwoody~\cite{MR807066},
a group has a context-free word problem if and only if it is virtually free.
These theorems suggest deep connections between group theory and formal language theory.

Involving both \(G\)-automata and word problems,
the following broad question was posed implicitly by Elston and Ostheimer~\cite{MR2064297}
and explicitly by Kambites~\cite{MR2259632}.
\begin{qu}\label[qu]{qu-grand}
	For a given group \(G\), is there any connection between the structural property of \(G\)
	and of the collection of groups whose word problems are accepted by \emph{non-deterministic} \(G\)-automata?
\end{qu}
\noindent Note that by \(G\)-automata, we always mean non-deterministic \(G\)-automata.
As for \emph{deterministic} \(G\)-automata, the following theorem is known.
\begin{thm}[Kambites~\cite{MR2259632}*{Theorem 1}, 2006]\label[thm]{theorem-deterministic}
	Let \(G\) and \(H\) be groups with \(H\) finitely generated.
	Then the word problem of \(H\) is accepted by a \emph{deterministic} \(G\)-automaton if and only if
	\(H\) has a finite index subgroup which embeds in \(G\).
\end{thm}

For non-deterministic \(G\)-automata, several results are known for specific types of groups.
For a free group \(F\) of rank \(\geq 2\),
it is known that a language is accepted by an \(F\)-automaton if and only if
it is context-free~\citelist{\cite{MR0152391}*{\textsc{Proposition} 2} \cite{MR2151422}*{Corollary 4.5} \cite{MR2482816}*{Theorem 7}}.
Combining with the Muller--Schupp theorem, the class of groups whose word problems are accepted by \(F\)-automata
is the class of virtually free groups.
The class of groups whose word problems are accepted by \((F \times F)\)-automata is exactly the class of
recursively presentable groups~\citelist{\cite{MR2151422}*{Corollary 3.5} \cite{MR2482816}*{Theorem 8} \cite{MR1807922}*{Theorem 10}}.
For the case where \(G\) is (virtually) abelian, the following result was shown by Elder, Kambites, and Ostheimer.
\begin{thm}[Elder, Kambites, and Ostheimer~\cite{MR2483126}, 2008]\label[thm]{theorem-EKO}
	\hfill
	\begin{enumerate}
		\item\label{item-Zn} Let \(H\) be a finitely generated group and \(n\) be a positive integer.
			Then the word problem of \(H\) is accepted by a \(\Z^n\)-automaton
			if and only if \(H\) is virtually free abelian of rank at most \(n\) \cite{MR2483126}*{Theorem 1}.
		\item\label{item-G-virtually-abelian} Let \(G\) be a virtually abelian group and \(H\) be a finitely generated group.
			Then the word problem of \(H\) is accepted by a \(G\)-automaton if and only if
			\(H\) has a finite index subgroup which embeds in \(G\) \cite{MR2483126}*{Theorem 4}.
	\end{enumerate}
\end{thm}
\noindent However, their proof is somewhat indirect in the sense that it depends on a deep theorem by Gromov~\cite{MR623534},
which states that every finitely generated group with polynomial growth function is virtually nilpotent.
In fact, their proof proceeds as follows.
Let \(H\) be a group whose word problem is accepted by a \(\Z^n\)-automaton.
They first develop some techniques to compute several bounds for linear maps and semilinear sets.
Then a map from \(H\) to \(\Z^n\) with certain geometric conditions is constructed to prove that \(H\) has polynomial growth function.
By Gromov's theorem, \(H\) is virtually nilpotent.
Finally, they conclude that \(H\) is virtually abelian, using some theorems about nilpotent groups and semilinear sets.
Because of the indirectness of their proof,
the embedding in \cref{theorem-EKO} \ref{item-G-virtually-abelian} is obtained only \emph{a posteriori}
and hence has no relation with the combinatorial structure of the \(G\)-automaton.

To our knowledge, there are almost no attempts so far to obtain explicit algebraic connections between \(G\) and \(H\),
where \(H\) is a group that has a word problem accepted by a \(G\)-automaton.
The only exception is the result due to Holt, Owens, and Thomas~\cite{MR2470538}*{\textsc{Theorem} 4.2},
where they gave a combinatorial proof to a special case of \cref{theorem-EKO} \ref{item-Zn}, for the case where \(n = 1\).
(In fact, their theorem is slightly stronger than \cref{theorem-EKO} \ref{item-Zn} for \(n = 1\)
because it is for \emph{non-blind} one-counter automata.
See also \cite{MR2483126}*{Section 7}.)

In this paper, we give a new, elementary, and purely combinatorial proof to \cref{theorem-EKO}.
\begin{thm}\label[thm]{main-theorem}
	Let \(G\) be an abelian group and \(H\) be a finitely generated group.
	Suppose that the word problem of \(H\) is accepted by a \(G\)-automaton \(A\).
	Then one can define a finite collection of monoids \((M(\mu, p))_{\mu, p}\), as in {\normalfont\cref{def-M}}, such that:
	\begin{enumerate}
		\item each \(M(\mu, p)\) consists of closed paths in \(A\) with certain conditions,
		\item each \(M(\mu, p)\) induces a group homomorphism \(f_{\mu, p}\)
			from a subgroup \(G(\mu, p)\) of \(G\) onto a subgroup \(H(\mu, p)\) of \(H\), and
		\item at least one of \(H(\mu, p)\)'s is a finite index subgroup of \(H\).
	\end{enumerate}
\end{thm}
\noindent For the implication from \cref{main-theorem} to \cref{theorem-EKO}, see \Cref{subsection-G-automata}.

Note that the direction of the group homomorphisms \(f_{\mu, p}\) in \cref{main-theorem} is opposite to
the embeddings in \cref{theorem-deterministic} and \cref{theorem-EKO} \ref{item-G-virtually-abelian}.
This direction seems more natural for the non-deterministic case;
this observation suggests the following question.

\begin{qu}\label[qu]{qu-hom}
	Let \(G\) and \(H\) be groups with \(H\) finitely generated.
	Suppose that the word problem of \(H\) is accepted by a \(G\)-automaton.
	Does there exist a group homomorphism from a subgroup of \(G\) onto a finite index subgroup of \(H\)? 
	If so, is it obtained combinatorially?
\end{qu}

Our \cref{main-theorem} is the very first step for approaching \cref{qu-hom}.
Note that an affirmative answer to \cref{qu-hom} would generalize \cref{theorem-deterministic}.

\section{Preliminaries}

\subsection{Words, subwords, and scattered subwords}

For a set \(\Sigma\), we write \(\Sigma^*\) for the free monoid generated by \(\Sigma\),
i.e., the set of \emph{words} over \(\Sigma\).
For a word \(u = a_1 a_2 \dotsm a_n \in \Sigma^*\) (\(n \geq 0, a_i \in \Sigma\)),
the number \(n\) is called the \emph{length} of \(u\), which is denoted by \(\len{u}\).
For two words \(u, v \in \Sigma^*\), the \emph{concatenation} of \(u\) and \(v\) are denoted by \(u \cdot v\), or simply \(u v\).
The identity element of \(\Sigma^*\) is the \emph{empty word}, denoted by \(\varepsilon\), which is the unique word of length zero.
For an integer \(n \geq 0\), the \(n\)-fold concatenation of a word \(u \in \Sigma^*\) is denoted by \(u^n\).
For an integer \(n > 0\), we write \(\Sigma^{< n}\) for the set of words of length less than \(n\).

A word \(u \in \Sigma^*\) is a \emph{subword} of a word \(v \in \Sigma^*\), denoted by \(u \subword v\),
if there exist two words \(u_1, u_2 \in \Sigma^*\) such that \(u_1 u u_2 = v\).
A word \(u \in \Sigma^*\) is a \emph{scattered subword} of a word \(v \in \Sigma^*\), denoted by \(u \scsubword v\),
if there exist two finite sequences of words \(u_1, u_2, \dotsc, u_n \in \Sigma^*\) (\(n \geq 0\))
and \(v_0, v_1, \dotsc, v_n \in \Sigma^*\) such that \(u = u_1 u_2 \dotsm u_n\) and \(v = v_0 u_1 v_1 u_2 v_2 \dotsm u_n v_n\).
That is, \(v\) is obtained from \(u\) by inserting some words.
Note that the two binary relations \(\subword\) and \(\scsubword\) are both partial orders on \(\Sigma^*\).

\subsection{Word problem for groups}\label{subsection-word-problem}

Let \(H\) be a finitely generated group.
A \emph{choice of generators} for \(H\) is a surjective monoid homomorphism \(\rho\)
from the free monoid \(\Sigma^*\) on a finite alphabet \(\Sigma\) onto \(H\).
The \emph{word problem} of \(H\) with respect to \(\rho\), denoted by \(\WP_\rho(H)\),
is the set of words in \(\Sigma^*\) mapped to the identity element \(1_H\) of \(H\) via \(\rho\),
i.e., \(\WP_\rho(H) = \rho^{-1}(1_H)\).

Although the word problem \(\WP_\rho(H)\) depends on the choice of generators \(\rho\), this does not cause problems:
\begin{lem}[e.g., \cite{MR2132375}*{Lemma 1}]
	Let \(\mathcal{C}\) be a class of languages closed under inverse homomorphisms and let \(H\) be a finitely generated group.
	Then \(\WP_\rho(H) \in \mathcal{C}\) for \emph{some} choice of generators \(\rho\)
	if and only if \(\WP_\rho(H) \in \mathcal{C}\) for \emph{any} choice of generators \(\rho\).
	\qed
\end{lem}
\noindent This is the reason why we use ``\emph{the} word problem of \(H\)'' rather than ``\emph{a} word problem of \(H\).''

\subsection{Graphs and paths}

A \emph{graph} is a \(4\)-tuple \((V, E, \source, \target)\), where \(V\) is the set of vertices, \(E\) is the set of (directed) edges,
\(\source\colon E \to V\) and \(\target\colon E \to V\) are functions
assigning to every edge \(e \in E\) the \emph{source} \(\source(e) \in V\) and the \emph{target} \(\target(e) \in V\), respectively.
A graph is \emph{finite} if it has only finitely many vertices and edges.

A \emph{path} (of \emph{length} \(n\)) in a graph \(\Gamma = (V, E, \source, \target)\) is a word
\(e_1 e_2 \dotsm e_n \in E^*\) (\(n \geq 0\)) of edges \(e_i \in E\)
such that \(\target(e_i) = \source(e_{i + 1})\) for \(i = 1, 2, \dotsc, n - 1\).
We usually use Greek letters to denote paths in a graph.
For a non-empty path \(\omega = e_1 e_2 \dotsm e_n \in E^*\), the source and the target of \(\omega\) are defined as
\(\source(\omega) = \source(e_1)\) and \(\target(\omega) = \target(e_n)\), respectively.
If \(\omega = e_1 e_2 \dotsm e_n\) and \(\omega' = e_1' e_2' \dotsm e_k'\) are non-empty paths
such that \(\target(\omega) = \source(\omega')\), or at least one of \(\omega\) and \(\omega'\) is empty,
then the concatenation of \(\omega\) and \(\omega'\), denoted by \(\omega \cdot \omega'\) or \(\omega \omega'\),
is the path \(e_1 e_2 \dotsm e_n e_1' e_2' \dotsm e_k'\) of length \(n + k\), i.e., the concatenation as words.
A path \(\omega\) in \(\Gamma\) is \emph{closed} if \(\source(\omega) = \target(\omega)\), or \(\omega = \varepsilon\).
For a closed path \(\sigma\) and an integer \(n \geq 0\), we write \(\sigma^n\) for the \(n\)-fold concatenation of \(\sigma\).

For a graph \(\Gamma = (V, E, \source, \target)\), an \emph{edge-labeling function} is a function \(\Label\) from \(E\) to a set \(M\).
If \(M\) is a monoid and \(\omega = e_1 e_2 \dotsm e_n\) is a path in \(\Gamma\),
then the label of \(\omega\) is defined as \(\Label(\omega) = \Label(e_1) \Label(e_2) \dotsm \Label(e_n)\)
via the multiplication of \(M\).

\subsection{\texorpdfstring{\(G\)}{G}-automata}\label{subsection-G-automata}

For a group \(G\), a (non-deterministic) \emph{\(G\)-automaton} over a finite alphabet \(\Sigma\)
is defined as a \(5\)-tuple \((\Gamma, \Label_G, \Label_\Sigma, \pinit, \pter)\), where
\(\Gamma = (V, E, \source, \target)\) is a finite graph,
\(\Label_G\colon E \to G\) and \(\Label_\Sigma\colon E \to \Sigma^*\) are edge-labeling functions,
\(\pinit \in V\) is the \emph{initial vertex}, and
\(\pter \in V\) is the \emph{terminal vertex}.
For simplicity, we assume that \(\Label_\Sigma(e) \in \Sigma \cup \{\varepsilon\}\) for each \(e \in E\).
(Note that this assumption does not decrease the accepting power of \(G\)-automata.
Indeed, if necessary, one can subdivide an edge \(e\) with labels \(\Label_\Sigma(e) = u v, \Label_G(e) = g\)
into two new edges \(e_1, e_2\) with labels \(\Label_\Sigma(e_1) = u, \Label_G(e_1) = g\)
and \(\Label_\Sigma(e_2) = v, \Label_G(e_2) = 1_G\).)
An \emph{accepting path} in a \(G\)-automaton \(A = (\Gamma, \Label_G, \Label_\Sigma, \pinit, \pter)\)
is a path \(\alpha\) in \(\Gamma\) such that \(\source(\alpha) = \pinit\), \(\target(\alpha) = \pter\), and \(\Label_G(\alpha) = 1_G\)
(we consider that the empty path \(\varepsilon \in E^*\) is accepting if and only if \(\pinit = \pter\)).
We say that a path \(\omega\) in \(\Gamma\) is \emph{promising} if \(\omega\) is a subword of some accepting path in \(A\),
i.e., there exist two paths \(\omega_1, \omega_2 \in E^*\)
such that the concatenation \(\omega_1 \omega \omega_2 \in E^*\) is an accepting path in \(A\).
The \emph{language accepted} by a \(G\)-automaton \(A\), denoted by \(L(A)\),
is the set of all words \(u \in \Sigma^*\) such that \(u\) is the label of some accepting path in \(A\),
i.e., \(L(A) = \{\, \Label_\Sigma(\alpha) \in \Sigma^* \mid \text{\(\alpha\) is an accepting path in \(A\)} \,\}\).

\begin{prop}[e.g., \cite{MR2482816}*{Proposition 2}]
	For a group \(G\), the class of languages accepted by \(G\)-automata are closed under inverse homomorphisms.
	\qed
\end{prop}

Replacing the register group \(G\) by its finite index subgroup or finite index overgroup does not change the class of languages accepted by \(G\)-automata:
\begin{prop}[e.g., \cite{MR2483126}*{Proposition 8}]
	Let \(G\) be a group and \(H\) be a subgroup of \(G\).
	Then every language accepted by a \(H\)-automaton is accepted by a \(G\)-automaton.
	If \(H\) has finite index in \(G\), then the converse holds.
	\qed
\end{prop}
\noindent Since the word problem of \(H\) is trivially accepted by an \(H\)-automaton for every finitely generated group \(H\),
we obtain the following corollary.
\begin{cor}\label[cor]{cor-implication}
	\cref{main-theorem} implies \cref{theorem-EKO}.
	\qed
\end{cor}

\section{Proof of the main theorem}\label{section-proof}

Throughout this section, we fix an abelian group \(G\), a finitely generated group \(H\),
a choice of generators \(\rho\colon \Sigma^* \to H\),
and a \(G\)-automaton \(A = (\Gamma = (V, E, \source, \target), \Label_G, \Label_\Sigma, \pinit, \pter)\)
such that \(\WP_\rho(H) = L(A)\).
We write the group operation of \(G\) additively and \(0_G\) for the identity element of \(G\).

The following lemma is a starting point of our proof.
\begin{lem}\label[lem]{lem-well-def}
	Let \(\omega\) and \(\omega'\) be paths in \(\Gamma\)
	such that \(\source(\omega) = \source(\omega')\) and \(\target(\omega) = \target(\omega')\),
	and suppose that \(\omega\) is promising.
	Then \(\Label_G(\omega) = \Label_G(\omega')\) implies \(\rho(\Label_\Sigma(\omega)) = \rho(\Label_\Sigma(\omega'))\).
\end{lem}
\begin{proof}
	Since \(\omega\) is promising, there exist two paths \(\omega_1, \omega_2\) in \(\Gamma\)
	such that \(\omega_1 \omega \omega_2\) is an accepting path in \(A\).
	It follows from the assumption that \(\Label_G(\omega_1 \omega' \omega_2)
	= \Label_G(\omega_1) + \Label_G(\omega') + \Label_G(\omega_2)
	= \Label_G(\omega_1) + \Label_G(\omega) + \Label_G(\omega_2) = \Label_G(\omega_1 \omega \omega_2) = 0_G\),
	and \(\omega_1 \omega' \omega_2\) is also an accepting path in \(A\).
	That is, \(\Label_\Sigma(\omega_1 \omega \omega_2), \Label_\Sigma(\omega_1 \omega' \omega_2) \in \WP_\rho(H)\),
	and \(\rho(\Label_\Sigma(\omega_1)) \rho(\Label_\Sigma(\omega)) \rho(\Label_\Sigma(\omega_2)) = 1_H
	= \rho(\Label_\Sigma(\omega_1)) \rho(\Label_\Sigma(\omega')) \rho(\Label_\Sigma(\omega_2))\) in \(H\).
	Thus we have \(\rho(\Label_\Sigma(\omega)) = \rho(\Label_\Sigma(\omega'))\).
\end{proof}

\begin{defi}\label[defi]{def-minimal}
	An accepting path \(\alpha\) in \(A\) is \emph{minimal}
	if it is minimal with respect to the scattered subword relation \(\scsubword\) on \(E^*\).
	An accepting path \(\alpha\) in \(A\) \emph{dominates} a minimal accepting path \(\mu\) in \(A\) if \(\mu \scsubword \alpha\).
\end{defi}

A similar notion of minimal accepting paths can be found in \cite{https://doi.org/10.48550/arxiv.math/0606415}*{Section 4}.

\begin{rem}\label[rem]{rem-Higman}
	Note that, by Higman's lemma~\cite{MR49867}*{\textsc{Theorem} 4.4},
	the scattered subword relation \(\scsubword\) on \(\Sigma^*\) is a \emph{well-quasi-order}.
	In particular, there are only finitely many minimal accepting paths in \(A\),
	and every accepting path on \(A\) dominates some minimal accepting path in \(A\).
\end{rem}

\begin{defi}\label[defi]{def-pumpable}
	Let \(\mu = e_1 e_2 \dotsm e_n \in E^*\) (\(e_i \in E\)) be a minimal accepting path in \(A\).
	A closed path \(\sigma \in E^*\) in \(\Gamma\) is \emph{pumpable in \(\mu\)}
	if there exists an accepting path \(\alpha\) in \(A\) dominating \(\mu\) such that
	\(\alpha = \alpha_0 e_1 \alpha_1 e_2 \alpha_2 \dotsm e_n \alpha_n\)
	for some paths \(\alpha_0, \alpha_1, \dotsc, \alpha_n \in E^*\) in \(\Gamma\)
	and \(\sigma \subword \alpha_j\) for some \(j \in \{0, 1, \dotsc, n\}\).
\end{defi}

\begin{rems}
	\hfill
	\begin{enumerate}
		\item In \cref{def-pumpable}, each \(\alpha_i\) is a closed path in \(\Gamma\)
			and satisfies \(\Label_G(\alpha_0) + \Label_G(\alpha_1) + \dotsb + \Label_G(\alpha_n) = \Label_G(\alpha) = 0_G\)
			since \(\Label_G(\mu) = 0_G\) and \(G\) is abelian.
		\item Every closed path pumpable in a minimal accepting path \(\mu\) is promising.
	\end{enumerate}
\end{rems}

\begin{defi}\label[defi]{def-M}
	For a minimal accepting path \(\mu\) in \(A\) and a vertex \(p \in V\),
	define
	\[
		M(\mu, p) = \{\, \sigma \mid
		\text{\(\sigma\) is a closed path in \(\Gamma\) pumpable in \(\mu\)
		such that \(\source(\sigma) = p\), or \(\sigma = \varepsilon\)} \,\}.
	\]
\end{defi}

\begin{lem}\label[lem]{lem-monoid}
	Each \(M(\mu, p)\) is a monoid with respect to the concatenation operation, i.e.,
	\(\sigma_1, \sigma_2 \in M(\mu, p)\) implies \(\sigma_1 \sigma_2 \in M(\mu, p)\).
\end{lem}
\begin{proof}
	Since both \(\sigma_1\) and \(\sigma_2\) are pumpable in \(\mu = e_1 e_2 \dotsm e_n \in E^*\) (\(e_i \in E\)),
	there exist two accepting paths \(\alpha = \alpha_0 e_1 \alpha_1 e_2 \alpha_2 \dotsm e_n \alpha_n\) (\(\alpha_i \in E^*\))
	and \(\beta = \beta_0 e_1 \beta_1 e_2 \beta_2 \dotsm e_n \beta_n\) (\(\beta_i \in E^*\))
	such that \(\sigma_1 \subword \alpha_i\) and \(\sigma_2 \subword \beta_j\) for some \(i, j \in \{0, 1, \dotsc, n\}\).
	Then we have \(\alpha_i = \alpha_i' \sigma_1 \alpha_i''\) for some \(\alpha_i', \alpha_i'' \in E^*\)
	and \(\beta_j = \beta_j' \sigma_2 \beta_j''\) for some \(\beta_j', \beta_j'' \in E^*\).
	We may assume that \(i \leq j\).
	Since \(G\) is abelian, the merged path \(\gamma = (\alpha_0 \beta_0) e_1 (\alpha_1 \beta_1) e_2 (\alpha_2 \beta_2) \dotsm e_n (\alpha_n \beta_n)\)
	and its permutation
	\begin{equation}
		\gamma' = (\alpha_0 \beta_0) e_1 (\alpha_1 \beta_1) e_2 (\alpha_2 \beta_2) \dotsm
		e_i (\alpha_i' \sigma_1 \sigma_2 \alpha_i'') e_{i + 1} \dotsm
		e_j (\beta_j' \beta_j'') e_{j + 1} \dotsm e_n (\alpha_n \beta_n) \label{eq-monoid}
	\end{equation}
	are accepting paths in \(A\) (\Cref{fig-monoid}).
	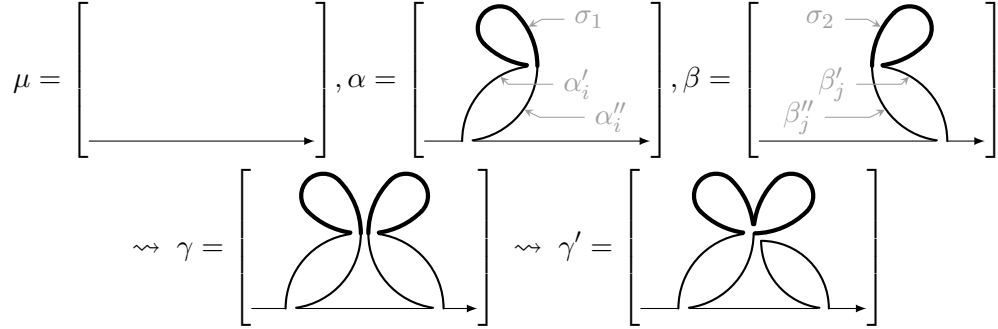
\begin{figure}
		\centering
		\begin{gather*}
			\mu = \mleft[
				\begin{tikzpicture}[scale = 1, baseline = (dummy.base)]
					\node (dummy) at (0, 0.8) {\phantom{x}};
					\draw[white, ultra thick, line cap = round] ({0 + cos(-135)}, {2 + sin(-135)})
						arc[start angle = -135, delta angle = -180, radius = {1 - 1 / sqrt(2)}];
					\draw[line cap = round, -latex] (-1.5, 0) -- (1.5, 0);
				\end{tikzpicture}
			\mright],
			\alpha = \mleft[
				\begin{tikzpicture}[scale = 1, baseline = (dummy.base)]
					\newcommand{\D}{0.15}
					\node (dummy) at (0, 0.8) {\phantom{x}};
					\draw[line cap = round] (-1.5, 0) -- (-1, 0);
					\draw[thick, line cap = round] (-1, 0) arc[start angle = 180, delta angle = -82, radius = 1];
					\draw[ultra thick, line cap = round] ({0 + cos(-135)}, {2 + sin(-135)}) arc[start angle = -135, delta angle = 82 - 45, radius = 1];
					\draw[ultra thick, line cap = round] ({0 + cos(-135)}, {2 + sin(-135)})
						arc[start angle = -135, delta angle = -180, radius = {1 - 1 / sqrt(2)}];
					\draw[ultra thick, line cap = round] ({-1 + cos(45)}, {1 + sin(45)}) arc[start angle = 45, delta angle = -45, radius = 1];
					\draw[thick, line cap = round] (0, 1) arc[start angle = 0, delta angle = -82, radius = 1];
					\draw[line cap = round, -latex] (-1 + \D, 0) -- (1.5, 0);
					\draw[gray!75!white, thin, stealth-] ({cos(120) + 0.01}, {sin(120) - 0.01}) -- ++(0.1, -0.1) -- ++(0.6, 0) node[right] {\(\alpha_i'\)};
					\draw[gray!75!white, thin, stealth-] ({-1 + cos(-35) + 0.01}, {1 + sin(-35) - 0.01}) -- ++(0.2, -0.1)
						-- ++(0.6, 0) node[right] {\(\alpha_i''\)};
					\draw[gray!75!white, thin, stealth-] ({-1 + cos(30) + 0.02}, {1 + sin(30) + 0.02}) -- ++({0.2 * cos(30)}, {0.2 * sin(30)})
						-- ++(0.3, 0) node[right] {\(\sigma_1\)};
				\end{tikzpicture}
			\mright],
			\beta = \mleft[
				\begin{tikzpicture}[scale = 1, baseline = (dummy.base)]
					\newcommand{\D}{0.15}
					\node (dummy) at (0, 0.8) {\phantom{x}};
					\draw[line cap = round] (-1.5, 0) -- (1 - \D, 0);
					\draw[thick, line cap = round] (0, 1) arc[start angle = 180, delta angle = 82, radius = 1];
					\draw[ultra thick, line cap = round] (0, 1) arc[start angle = 180, delta angle = -45, radius = 1];
					\draw[ultra thick, line cap = round] ({0 + cos(-45)}, {2 + sin(-45)})
						arc[start angle = -45, delta angle = 180, radius = {1 - 1 / sqrt(2)}];
					\draw[ultra thick, line cap = round] ({0 + cos(-45)}, {2 + sin(-45)}) arc[start angle = -45, delta angle = -82 + 45, radius = 1];
					\draw[thick, line cap = round] (1, 0) arc[start angle = 0, delta angle = 82, radius = 1];
					\draw[line cap = round, -latex] (1, 0) -- (1.5, 0);
					\draw[gray!75!white, thin, stealth-] ({cos(60) - 0.01}, {sin(60) - 0.01}) -- ++(-0.1, -0.1) -- ++(-0.6, 0) node[left] {\(\beta_j'\)};
					\draw[gray!75!white, thin, stealth-] ({1 + cos(180 + 35) - 0.01}, {1 + sin(180 + 35) - 0.01}) -- ++(-0.2, -0.1)
						-- ++(-0.6, 0) node[left] {\(\beta_j''\)};
					\draw[gray!75!white, thin, stealth-] ({1 + cos(180 - 30) - 0.02}, {1 + sin(180 - 30) + 0.02})
						-- ++({0.2 * cos(180 - 30)}, {0.2 * sin(180 - 30)}) -- ++(-0.3, 0) node[left] {\(\sigma_2\)};
				\end{tikzpicture}
			\mright]\\
			\leadsto \;
			\gamma = \mleft[
				\begin{tikzpicture}[scale = 1, baseline = (dummy.base)]
					\newcommand{\D}{0.05}
					\node (dummy) at (0, 0.8) {\phantom{x}};
					\draw[line cap = round] (-1.5, 0) -- (-1 - \D, 0);
					\begin{scope}[shift = {(-\D, 0)}]
						\draw[thick, line cap = round] (-1, 0) arc[start angle = 180, delta angle = -82, radius = 1];
						\draw[ultra thick, line cap = round] ({0 + cos(-135)}, {2 + sin(-135)}) arc[start angle = -135, delta angle = 82 - 45, radius = 1];
						\draw[ultra thick, line cap = round] ({0 + cos(-135)}, {2 + sin(-135)})
							arc[start angle = -135, delta angle = -180, radius = {1 - 1 / sqrt(2)}];
						\draw[ultra thick, line cap = round] ({-1 + cos(45)}, {1 + sin(45)}) arc[start angle = 45, delta angle = -45, radius = 1];
						\draw[thick, line cap = round] (0, 1) arc[start angle = 0, delta angle = -82, radius = 1];
					\end{scope}
					\draw[line cap = round] (-1 + 2 * \D, 0) -- (1 - 2 * \D, 0);
					\begin{scope}[shift = {(\D, 0)}]
						\draw[thick, line cap = round] (0, 1) arc[start angle = 180, delta angle = 82, radius = 1];
						\draw[ultra thick, line cap = round] (0, 1) arc[start angle = 180, delta angle = -45, radius = 1];
						\draw[ultra thick, line cap = round] ({0 + cos(-45)}, {2 + sin(-45)})
							arc[start angle = -45, delta angle = 180, radius = {1 - 1 / sqrt(2)}];
						\draw[ultra thick, line cap = round] ({0 + cos(-45)}, {2 + sin(-45)}) arc[start angle = -45, delta angle = -82 + 45, radius = 1];
						\draw[thick, line cap = round] (1, 0) arc[start angle = 0, delta angle = 82, radius = 1];
					\end{scope}
					\draw[line cap = round, -latex] (1 + \D, 0) -- (1.5, 0);
				\end{tikzpicture}
			\mright]
			\; \leadsto \;
			\gamma' = \mleft[
				\begin{tikzpicture}[scale = 1, baseline = (dummy.base)]
					\newcommand{\D}{0.1}
					\node (dummy) at (0, 0.8) {\phantom{x}};
					\draw[line cap = round] (-1.5, 0) -- (-1, 0);
					\draw[thick, line cap = round] (-1, 0) arc[start angle = 180, delta angle = -82, radius = 1];
					\draw[ultra thick, line cap = round] ({0 + cos(-135)}, {2 + sin(-135)}) arc[start angle = -135, delta angle = 82 - 45, radius = 1];
					\draw[ultra thick, line cap = round] ({0 + cos(-135)}, {2 + sin(-135)})
						arc[start angle = -135, delta angle = -180, radius = {1 - 1 / sqrt(2)}];
					\draw[ultra thick, line cap = round] ({-1 + cos(45)}, {1 + sin(45)}) arc[start angle = 45, delta angle = -38, radius = 1];
					\draw[ultra thick, line cap = round] ({1 + cos(135)}, {1 + sin(135)}) arc[start angle = 135, delta angle = 38, radius = 1];
					\draw[ultra thick, line cap = round] ({0 + cos(-45)}, {2 + sin(-45)})
						arc[start angle = -45, delta angle = 180, radius = {1 - 1 / sqrt(2)}];
					\draw[ultra thick, line cap = round] ({0 + cos(-45)}, {2 + sin(-45)}) arc[start angle = -45, delta angle = -44, radius = 1];
					\draw[thick, line cap = round] (0, 1) arc[start angle = 0, delta angle = -82, radius = 1];
					\draw[line cap = round] (-1 + 1.5 * \D, 0) -- (1 - 1.5 * \D, 0);
					\draw[thick, line cap = round] (0 + \D, 1 - \D) arc[start angle = 180, delta angle = 82, radius = 1 - \D];
					\draw[thick, line cap = round] (0 + \D, 1 - \D) arc[start angle = 90, delta angle = -90, radius = 1 - \D];
					\draw[line cap = round, -latex] (1, 0) -- (1.5, 0);
				\end{tikzpicture}
			\mright]
		\end{gather*}
		\caption{Construction of the accepting path \(\gamma'\) in \eqref{eq-monoid}}\label{fig-monoid}
	\end{figure}
\end{proof}

For each \(M(\mu, p)\), \cref{lem-monoid} allows us to define a surjective monoid homomorphism
\(f_{\mu, p}\colon M(\mu, p) \to \rho(\Label_\Sigma(M(\mu, p)))\) as the composition function \(\rho \circ \Label_\Sigma\).
By \cref{lem-well-def}, \(f_{\mu, p}\) induces a well-defined surjective monoid homomorphism
\(\bar{f}_{\mu, p}\colon \Label_G(M(\mu, p)) \to \rho(\Label_\Sigma(M(\mu, p)))\).
Let \(G(\mu, p)\) (resp.\ \(H(\mu, p)\)) denotes the subgroup of \(G\) generated by \(\Label_G(M(\mu, p))\)
(resp.\ the subgroup of \(H\) generated by \(\rho(\Label_\Sigma(M(\mu, p)))\)).
One can easily extend \(\bar{f}_{\mu, p}\) to a unique surjective group homomorphism
\(\tilde{f}_{\mu, p}\colon G(\mu, p) \to H(\mu, p)\).
The remaining task is to prove that at least one of the \(H(\mu, p)\)'s is a finite index subgroup of \(H\).

\begin{lem}\label[lem]{lem-downward}
	Each \(M(\mu, p)\) is downward closed with respect to \(\scsubword\),
	i.e., if \(\sigma\) is an element of \(M(\mu, p)\) and \(\tau\) is a closed path in \(\Gamma\) with \(\source(\tau) = p\)
	such that \(\tau \scsubword \sigma\), then \(\tau \in M(\mu, p)\).
\end{lem}
\begin{proof}
	Suppose that \(\tau = e_1' e_2' \dotsm e_k'\) (\(k \geq 0, e_i' \in E\))
	and \(\sigma = \sigma_0 e_1' \sigma_1 e_2' \dotsm e_k' \sigma_k\) (\(\sigma_i \in E^*\)).
	Each \(\sigma_i\) is a closed path in \(\Gamma\).
	Since, by \cref{lem-monoid}, \(\sigma^2\) is pumpable in \(\mu = e_1 e_2 \dotsm e_n\) (\(n \geq 0, e_i \in E\)),
	there exists an accepting path \(\alpha = \alpha_0 e_1 \alpha_1 e_2 \alpha_2 \dotsm e_n \alpha_n\) dominating \(\mu\)
	such that \(\sigma^2 \scsubword \alpha_i\) for some \(i \in \{0, 1, \dotsc, n\}\).
	If \(\alpha_i = \alpha_i' \sigma^2 \alpha_i''\), then the path
	\begin{equation}
		\gamma = \alpha_0 e_1 \alpha_1 e_2 \alpha_2 \dotsm e_i
		(\alpha_i' \cdot \tau \cdot (\sigma_0^2 e_1' \sigma_1^2 e_2' \sigma_2^2 \dotsm e_k' \sigma_k^2) \cdot \alpha_i'')
		e_{i + 1} \dotsm e_n \alpha_n \label{eq-downward}
	\end{equation}
	is an accepting path in \(A\) (\Cref{fig-downward}).
	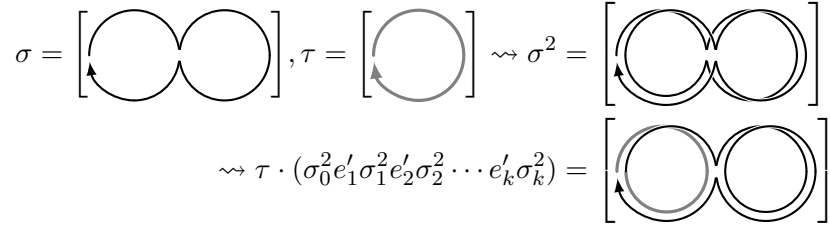
\begin{figure}
		\centering
		\begin{align*}
			\sigma = \mleft[
				\begin{tikzpicture}[scale = 0.6, baseline = (dummy.base)]
					\node (dummy) at (0, 0) {\phantom{x}};
					\draw[thick, line cap = round] (-2, 0) arc[start angle = 180, delta angle = -172, radius = 1];
					\draw[thick, line cap = round] (2, 0) arc[start angle = 0, delta angle = 172, radius = 1];
					\draw[thick, line cap = round] (2, 0) arc[start angle = 0, delta angle = -172, radius = 1];
					\draw[thick, line cap = round] (-1, -1) arc[start angle = -90, delta angle = 82, radius = 1];
					\draw[thick, line cap = round, -latex] (-1, -1) arc[start angle = -90, delta angle = -82, radius = 1];
				\end{tikzpicture}
			\mright],
			\tau = \mleft[
				\begin{tikzpicture}[scale = 0.6, baseline = (dummy.base)]
					\node (dummy) at (-0.5, 0) {\phantom{x}};
					\draw[very thick, gray, line cap = round, -latex] (-2, 0) arc[start angle = 180, delta angle = -360 + 8, radius = 1];
				\end{tikzpicture}
			\mright]
			\leadsto
			\sigma^2 &= \mleft[
				\begin{tikzpicture}[scale = 0.6, baseline = (dummy.base)]
					\node (dummy) at (0, 0) {\phantom{x}};
					\draw[thick, line cap = round] (-2, 0) arc[start angle = 180, delta angle = -172, radius = 1];
					\draw[thick, line cap = round] (2, 0) arc[start angle = 0, delta angle = 172, radius = 1];
					\draw[thick, line cap = round] (2, 0) arc[start angle = 0, delta angle = -172, radius = 1];
					\draw[thick, line cap = round] (-0.9, -0.9) arc[start angle = -90, delta angle = 82, radius = 0.9];
					\draw[thick, line cap = round] (-0.9, -0.9) arc[start angle = -90, delta angle = -90, radius = 0.9];
					\draw[preaction = {draw = white, ultra thick}, thick, line cap = round] (-1.8, 0) arc[start angle = 180, delta angle = -172, radius = 1]
						(2.2, 0) arc[start angle = 0, delta angle = 172, radius = 1];
					\draw[preaction = {draw = white, ultra thick}, thick, line cap = round] (2.2, 0) arc[start angle = 0, delta angle = -172, radius = 1]
						(-0.9, -1.1) arc[start angle = -90, delta angle = 82, radius = 1.1];
					\draw[thick, line cap = round, -latex] (-0.9, -1.1) arc[start angle = -90, delta angle = -82, radius = 1.1];
				\end{tikzpicture}
			\mright]\\
			\leadsto
			\tau \cdot (\sigma_0^2 e_1' \sigma_1^2 e_2' \sigma_2^2 \dotsm e_k' \sigma_k^2) &= \mleft[
				\begin{tikzpicture}[scale = 0.6, baseline = (dummy.base)]
					\node (dummy) at (0, 0) {\phantom{x}};
					\draw[very thick, gray, line cap = round] (-2.2, 0) arc[start angle = 180, delta angle = -180, radius = 1];
					\draw[very thick, gray, line cap = round] (-0.2, 0) arc[start angle = 0, delta angle = -180, radius = 0.9];
					\draw[preaction = {draw = white, ultra thick}, thick, line cap = round] (-2, 0) arc[start angle = 180, delta angle = -172, radius = 1];
					\draw[thick, line cap = round] (2, 0) arc[start angle = 0, delta angle = 172, radius = 1];
					\draw[thick, line cap = round] (2, 0) arc[start angle = 0, delta angle = -180, radius = 0.9];
					\draw[preaction = {draw = white, ultra thick}, thick, line cap = round] (0.2, 0) arc[start angle = 180, delta angle = -180, radius = 1];
					\draw[thick, line cap = round] (2.2, 0) arc[start angle = 0, delta angle = -172, radius = 1.1];
					\draw[thick, line cap = round] (-1.1, -1.1) arc[start angle = -90, delta angle = 82, radius = 1.1];
					\draw[thick, line cap = round, -latex] (-1.1, -1.1) arc[start angle = -90, delta angle = -82, radius = 1.1];
				\end{tikzpicture}
			\mright]
		\end{align*}
		\caption{Construction of the path \(\tau \cdot (\sigma_0^2 e_1' \sigma_1^2 e_2' \sigma_2^2 \dotsm e_k' \sigma_k^2)\)
			in \eqref{eq-downward}}\label{fig-downward}
	\end{figure}
\end{proof}

\begin{lem}\label[lem]{lem-shrink}
	Let \(\sigma \in M(\mu, p)\) and \(\omega \subword \sigma\) be a path.
	Then there exist two paths \(\omega_1, \omega_2 \in E^{< \card{V}}\)
	such that \(\omega_1 \omega \omega_2 \in M(\mu, p)\).
\end{lem}
\begin{proof}
	Let \((\omega_1, \omega_2) \in E^* \times E^*\) be a pair of two paths
	such that \(\omega_1 \omega \omega_2 \in M(\mu, p)\) and \(\max\{\len{\omega_1}, \len{\omega_2}\}\) is minimum.
	Such a pair exists since \(\omega \subword \sigma \in M(\mu, p)\).
	Suppose the contrary that \(\max\{\len{\omega_1}, \len{\omega_2}\} \geq \card{V}\), say \(\len{\omega_1} \geq \card{V}\).
	By the pigeonhole principle, \(\omega_1\) must visit some vertex \(p \in V\) at least twice.
	That is, there exist three paths \(\alpha, \beta, \gamma\)
	such that \(\omega_1 = \alpha \beta \gamma\) and \(\beta\) is a non-empty closed path.
	Now we have \(\alpha \gamma \omega \omega_2 \scsubword \omega_1 \omega \omega_2 \in M(\mu, p)\),
	and \cref{lem-downward} implies \(\alpha \gamma \omega \omega_2 \in M(\mu, p)\),
	which contradicts the minimality of \((\omega_1, \omega_2)\).
\end{proof}

\begin{proof}[Proof of \cref{main-theorem}]
	Let \(h \in H\) and fix a word \(v \in \Sigma^*\) such that \(\rho(v) = h\).
	There exists a word \(\bar{v} \in \Sigma^*\) such that \(v \bar{v} \in \WP_\rho(H)\).
	Define
	\[
		N = 1 + \max\{\, \len{\mu} \mid \text{\(\mu \in E^*\) is a minimal accepting path in \(A\)} \,\},
	\]
	and then \(N < \infty\) by \cref{rem-Higman}.
	Since \((v \bar{v})^N \in \WP_\rho(H)\),
	there exists an accepting path
	\begin{equation}
		\alpha = \omega_1 \bar{\omega}_1 \omega_2 \bar{\omega}_2 \dotsm \omega_N \bar{\omega}_N \label{eq-v}
	\end{equation}
	in \(A\) such that \(\Label_\Sigma(\omega_i) = v\) and \(\Label_\Sigma(\bar{\omega}_i) = \bar{v}\) for \(i = 1, 2, \dotsc, N\).
	Let \(\mu = e_1 e_2 \dotsm e_n\) (\(e_i \in E\)) be a minimal accepting path such that \(\alpha\) dominates \(\mu\).
	Then we have another decomposition
	\begin{equation}
		\alpha = \alpha_0 e_1 \alpha_1 e_2 \alpha_2 \dotsm e_n \alpha_n \label{eq-mu}
	\end{equation}
	for some closed paths \(\alpha_0, \alpha_1, \dotsc, \alpha_n \in E^*\).
	Since \(N > n\) and each \(e_i\) in the decomposition \eqref{eq-mu} is contained in
	at most one \(\omega_i\) in the decomposition \eqref{eq-v},
	at least one of the \(\omega_i\)'s is disjoint from all \(e_i\)'s,
	i.e., there exist \(i \in \{1, 2, \dotsc, N\}\) and \(j \in \{0, 1, \dotsc, n\}\) such that \(\omega_i \subword \alpha_j\).
	Since \(\alpha_j\) is a pumpable closed path in \(\mu\), \(\alpha_j\) is an element of \(M(\mu, \source(\alpha_j))\).
	By \cref{lem-shrink}, there exist \(\alpha_j', \alpha_j'' \in E^{< \card{V}}\)
	such that \(\alpha_j' \omega_i \alpha_j'' \in M(\mu, \source(\alpha_j))\).
	Then we have \(\len{\Label_\Sigma(\alpha_j')}, \len{\Label_\Sigma(\alpha_j'')} < \card{V}\) and
	\(\rho(\Label_\Sigma(\alpha_j')) \rho(\Label_\Sigma(\omega_i)) \rho(\Label_\Sigma(\alpha_j'')) \in H(\mu, \source(\alpha))\),
	hence
	\[
		h = \rho(v) = \rho(\Label_\Sigma(\omega_i))
		\in \rho(\Label_\Sigma(\alpha_j'))^{-1} H(\mu, \source(\alpha)) \rho(\Label_\Sigma(\alpha_j''))^{-1}.
	\]

	From the above argument, we obtain
	\[
		H = \bigcup \mleft\{\, h_1^{-1} H(\mu, p) h_2^{-1} \setmid
		\begin{gathered}
			\text{\(\mu\) is a minimal accepting path in \(A\),}\\
			\text{\(p \in V\), and \(h_1, h_2 \in \rho(\Sigma^{< \card{V}})\)} \,
		\end{gathered}
		\mright\},
	\]
	where the right-hand side is a finite union of cosets of \(H\) by \cref{rem-Higman}.
	Thus, by B.\ H.\ Neumann's lemma~\cite{MR62122}*{(4.1) \textsc{Lemma} and (4.2)},
	at least one of \(H(\mu, p)\)'s has finite index in \(H\).
\end{proof}

\section*{Acknowledgements}

The author would like to thank Ryoma Sin'ya for helpful comments and encouragement.

\bibliographystyle{plain}
\bibliography{refs}

\end{document}